\documentclass{amsproc}

\usepackage{amssymb}

\usepackage{graphicx}
\usepackage{tikz}
\usepackage{}

\newtheorem{theorem}{Theorem}[section]
\newtheorem{lemma}[theorem]{Lemma}
\newtheorem{corollary}[theorem]{Corollary}

\theoremstyle{definition}

\theoremstyle{remark}

\numberwithin{equation}{section}

\begin{document}

\title{}


%    author one information
\author{}
\address{}
\curraddr{}
\email{}
\thanks{}

%    author two information
\author{}
\address{}
\curraddr{}
\email{}
\thanks{}

%    The 2010 edition of the Mathematics Subject Classification is
%    now available.  If you are citing a classification from the
%    new scheme, use the following input coding instead.
%\subjclass[2010]{Primary }
\title{$k$-protected Vertices In Unlabeled Rooted Plane Trees}
\author{Keith Copenhaver}
\date{}

%\date{}
\maketitle

\begin{abstract}
We find a simple, closed formula for the proportion of vertices which are \textit{k}-protected in all unlabeled rooted plane trees on $n$ vertices. We also find that, as $n$ goes to infinity,  the average rank of a random vertex in a tree of size $n$ approaches 0.727649, and the average rank of the root of a tree of size $n$ approaches 1.62297.

%\subclass{05A15 \and 05A16 \and 05C05}


\end{abstract}

\pagebreak
\section{Introduction}
Unlabeled rooted plane trees are one of the simplest tree structures. They are used to model any network with a point of origin.  We will ask the question, if you start at a leaf and take only upward steps, how far might we expect to go to reach a vertex?

The modern world is full of different kinds of networks. In every network it could be advantageous or disadvantageous to have highly protected vertices. For example, in many online communities, people must be invited in order to participate. If we were to create a tree of users where each person is connected to the person who invited them, then each time a user adds a member they once again become 1-protected. Thus a high level of protection most likely indicates that someone is not regularly bringing in new members.

In a network that must remain secure, it is desirable to make it difficult to reach the root, so we would like for the root to be highly protected (another question we will address directly), but we also would like nodes in general to have a high level of protection. For example, in a computer network with different levels of access, even if the network is very tall (in other words, the lowest entry point is many steps from the root), if it has a leaf at a high level, the root will still be far more accessible than the height of the network would imply.

\begin{figure}[h]
\begin{tikzpicture}
[level distance=0.8cm,
  level 1/.style={sibling distance=2cm},
  level 2/.style={sibling distance=1cm}]
  \node[circle,draw] {2}
    child {node[circle, draw] {1}
    child {node[circle, draw] {0}}
    		child{node[circle, draw] {1}
    			child{node[circle, draw] {0}}
    		}
    	}
	child {node[circle, draw] {4}
		child{node[circle, draw] {3}
		child{node[circle, draw] {2}
		child{node[circle, draw] {1}
		child{node[circle, draw] {0}
		}
		}
		}	
	}
}
;

\end{tikzpicture}	
\caption{}
\end{figure}

We will call a vertex \textit{k-protected} if we must take at least $k$ downward steps to reach any leaf. In Figure 1, each vertex is labeled by the highest $k$ for which it is $k$-protected (any vertex that is 4-protected is also 3-protected), this is also known as the \textit{rank}. The number of $k$-protected vertices was explored for this class of tree with $k=2$ by Cheon and Shapiro \cite{cheon}. The topic of protected vertices has been examined for random recursive trees by Mahmoud and Ward  \cite{mahmoudward}, $k$-ary trees by Mansour \cite{mansour}, binary search trees by B{\'o}na  \cite{bonasearch}, random phylogenetic trees by B{\'o}na and Flajolet \cite{bonaflajolet}, several types of random trees by Devroye and Janson \cite{devroyeandjanson}, and digital search trees by Du and Prodinger \cite{duandprodinger}.

Throughout this paper any reference to a tree will specifically mean an unlabeled rooted plane tree.

We will also freely use a few facts. The number of such trees on $n$ vertices, $t(n)$ is the $(n-1)$st Catalan number, and the ordinary generating function of the number of such trees on $n$ vertices $t_n$ is $\sum_{n=1}^\infty t(n) x^n=T(x)= \frac{1- \sqrt{1-4x}}2$. The number of all vertices of trees of size $n$, $v(n)$, is the $(n-1)$st central binomial coefficient, and they have the ordinary generating function $\sum_{n=1}^\infty v(n)x^n=V(x) = \frac{x}{\sqrt{1-4x}}$.

\section{Grafting}
Grafting is a process in horticulture where a branch from one tree is removed and a branch from a different tree is inserted in its place. We will do something very similar with the following lemma.
\begin{lemma}
Let $R_k(x)=\sum_{n =1}^\infty r_k(n)\, x^n$ be the ordinary generating function for $r_k(n)$, the number of rooted plane trees on $n$ vertices whose root is $k$-protected. Let $L(x)=\sum_{n=0}^\infty l(n)\, x^n$ be the ordinary generating function for $l(n)$, the number of leaves in all rooted plane trees of size $n+1$ (or, equivalently, all rooted plane trees with $n$ edges). Let $T_k(x)= \sum_{n=1}^\infty t_k(n)\, x^n$ be the ordinary generating function for $t_k(n)$, the number of $k$-protected vertices in all trees of size $n$. Then $T_k(x)=L(x) \cdot R_k(x)$.
\end{lemma}
\begin{proof}
Let $0<m \leq n$ We can count the number of $k$-protected vertices on a tree with $n$ vertices by choosing a tree on $n-m+1$ vertices, removing a specific leaf, then replacing that leaf with a tree whose root is $k$-protected. We can choose that leaf in one of $l(n-m)$ ways and the tree in $r_k(m)$ ways, so the result follows by the product formula.
\end{proof}
To apply this, of course, we need $L(x)$. Let $T(x, y)= \sum_{n=0}^\infty t_{n, m}x^n y^m$ where $t_{n, m}$ is the number of trees on $n$ vertices with $m$ leaves. We have that $T(x, y)$ satisfies the functional equation $T(x, y)= xy+x\left(\frac{T(x, y)}{1-T(x, y)} \right )$, and solving for $T(x, y)$, we have $T(x, y)=\frac 1 2 (1 - x + x y - \sqrt{-4 x y + (-1 + x - x y)^2})$. The function we want, $L(x)$, will be equal to $\left. \frac{ \partial}{\partial y} T(x, y) \right |_{y=1}$. Thus 
\begin{equation}
L(x)=\frac 1 2 \left(1 + \frac{1}{\sqrt{1 - 4 x}} \right).
\end{equation}
To find expressions for $R_k(x)$, we first observe that since a tree with a 1-protected root is simply a non-empty sequence of trees, we have $R_1(x)=\frac{xT(x)}{1-T(x)}$ where $T(x)=\frac {1- \sqrt{1-4x}}{2}$ giving 
$R_1(x)=\frac{1-2x-\sqrt{1-4x}}2.$
Further, this can be iterated since a root is $k$-protected if and only if it is a non-empty sequence of trees whose roots are $(k-1)$-protected. Thus we have the recursion $R_k(x)=x \cdot \frac{R_{k-1}(x)}{1-R_{k-1}(x)}$.

\begin{theorem}
For all $k \geq 2$,
\begin{equation}
R_k(x)= \frac{x^{k-2}\left(n_k(x)-\sqrt{1-4x} \right)}{2d_k(x)},
\end{equation}
where $n_k(x)$ and $d_k(x)$ are polynomials defined as follows: for all $k \geq 2$, $n_k(x)= 1 - 2x -2x^2 - ... - 2x^k,$
$d_2=2+x$, and for all $k \geq 3$
\begin{equation}
d_k(x)=\sum_{i=0}^{k-3} (i+1)x^i  +\sum_{i=k-2}^{2k-3} (2k+2-i)x^i.
\end{equation}
\end{theorem}
For some numerical justification, this gives the series expansions $R_2(x) = x^3 + 2x^4+6x^5+18x^6+...$ and $R_3(x) = x^4+2x^5+6x^6+...$, each of which are accurate up to trees of size 6 by examination.
To prove this theorem we will need some purely computational lemmas.

\begin{lemma} For all $k \geq 2$,
\begin{equation}
d_{k+1}(x)=d_k(x)-x^{k-2}n_k(x)+x^{2k-1}.
\end{equation}
\end{lemma}
\begin{proof} We will proceed by induction on $k$.

For the base case, if $k=2$, then 
$$d_3(x)=1+3x+2x^2+x^3=2+x-(1-2x-2x^2)+x^3=d_2(x)-x^{2-2}n_2(x)+x^{2(2)-1}.$$

For the induction step, we assume that the statement holds for $d_k(x)$, so from (2.2),

\begin{align*}
d_{k+1}(x)&=\sum_{i=0}^{(k+1)-3} (i+1)x^i  +\sum_{i=k-1}^{2(k+1)-3} (2(k+1)+2-i)x^i \\
%&=\sum_{i=0}^{k+1-3} (i+1)x^i  +\sum_{i=k-1}^{2k-2} (2(k+1)+2-i)x^i+x^{2k-1}\\
%&=\sum_{i=0}^{k-3} (i+1)x^i  + kx^{k-2} +\sum_{i=k-1}^{2k-3} (2k+2-i)x^i -(x^{k-2}-2\sum_{i=k-1}^{2k-2} x^i)+x^{2k-1} \\
&=\sum_{i=0}^{k-3} (i+1)x^i  +\sum_{i=k-2}^{2k-3} (2k+2-i)x^i -x^{k-2}(1-2\sum_{i=1}^k x^i)+x^{2k-1} \\
&=d_k(x)-x^{k-2}n_k(x)+x^{2k-1}.
\end{align*}

\end{proof}
\begin{lemma} If $k \geq 2$, then
\begin{equation} 
n^2_k(x)-(1-4x)=4x^3d_k(x).
\end{equation}
\end{lemma}
\begin{proof} The expansion of $n^2_k(x)$ splits nicely into 3 parts as follows:

The first two terms will be $1-4x$ for all $k \geq 1$.

For all $1 < i \leq k$ we will have a two copies of $-2x^i$ and $i-1$ copies of $4x^i$ giving a net total of $(i-2)x^i$. This means there will be no $x^2$ term.

For all $ k< i \leq 2k$ we will have no negative terms, and for each term we will have $2k-1-i$ copies of $4x^i$ giving us

%\begin{equation}
$$n^2_k(x)=1-4x+\sum_{i=3}^k (i-2)4x^i+ \sum_{i=k+1}^{2k} (2k-1-i)4x^i $$
$$=1-4x +4x^3 \left( \sum_{i=0}^{k-3} (i+1)x^i + \sum_{i=k-2}^{2k-3} (2k+2-i)x^i \right).$$
%\end{equation}
\end{proof}

\begin{proof}[of Theorem 1]
We once again proceed by induction on $k$.

For the base case, if $k=2$, then, after clearing denominators and multiplying by the conjugate of the denominator,
%\begin{equation}
%$$R_2(x)=\frac{x \cdot R_1(x)}{1-{R_1(x)}}=\frac{\frac{x \left ( 1-2x-\sqrt{1-4x} \right)}{2}}{\frac 2 2 - \frac{ 1-2x-\sqrt{1-4x} }{2}}=\frac{x \left(1-2x-\sqrt{1-4x}\right)}{1+2x+\sqrt{1-4x}} $$
%$$=\frac{x \left(1-2x-\sqrt{1-4x}\right)}{1+2x+\sqrt{1-4x}} \cdot \frac{1+2x-\sqrt{1-4x}}{1+2x-\sqrt{1-4x}}=\frac{2x\left(1-2x-2x^2 -\sqrt{1-4x}\right)}{4x(2+x)}$$
$$R_2(x)=\frac{x \cdot R_1(x)}{1-{R_1(x)}}=\frac{2x\left(1-2x-2x^2 -\sqrt{1-4x}\right)}{4x(2+x)}= \frac{x^0\left(n_2(x)-\sqrt{1-4x}\right)}{2d_2(x)}.$$
%\end{equation}
Now, assuming that the statement holds for $R_k(x)$, and after clearing denominators and multiplying by the conjugate of the denominator,
%\begin{equation}
$$R_{k+1}(x)=\frac{x \cdot R_k}{1-R_k}$$%= \frac{x \cdot \frac{x^{k-2}\left( n_k(x) -\sqrt{1-4x}\right)}{2d_k(x)}}{1-\frac{x^{k-2}\left( n_k(x) -\sqrt{1-4x}\right)}{2d_k(x)}}$$
%$$=\frac{x^{k-1}(n_k(x)-\sqrt{1-4x})}{2d_k(x)-x^{k-2}n_k(x)-x^{k-2}\sqrt{1-4x}} $$
%&=\frac{x^{k-1}(n_k(x)-\sqrt{1-4x})}{(2d_k(x)-x^{k-2}n_k(x))-x^{k-2}\sqrt{1-4x}}\cdot \frac{(2d_k(x)-x^{k-2}n_k(x))+x^{k-2}\sqrt{1-4x}}{(2d_k(x)-x^{k-2}n_k(x))+x^{k-2}\sqrt{1-4x}} \\
$$=\frac{x^{k-1} \left( 2d_k(x) n_k(x)-x^{k-2}n^2_k(x)-2d_k(x)\sqrt{1-4x}+x^{n-2}(1+4x)\right)}{4d^2_k(x)-4x^{k-2}d_k(x)n_k(x)+x^{2k-4}n^2_k(x)-x^{2k-4}(1-4x)} $$
$$=\frac{x^{k-1} \left( 2d_k(x) n_k(x)-x^{k-2}(4x^3d_k(x))-2d_k(x)\sqrt{1-4x}\right)}{4d^2_k(x)-4x^{k-2}d_k(x)n_k(x)+x^{2k-4}(4x^3d_k(x))}$$
$$=\frac{x^{k-1} \left( 2 n_k(x)-4x^{k+1}-2\sqrt{1-4x}\right)}{4d_k(x)-4x^{k-2}n_k(x)+4x^{2k-1}}=\frac{x^{k-1} \left(n_{k+1}(x)-\sqrt{1-4x}\right)}{2d_{k+1}(x)}.$$
%\end{equation}
\end{proof}

\section{Asymptotics}
We will say that $a_n \sim b_n$ if $\displaystyle{\lim_{n \rightarrow \infty} \frac {a_n} {b_n} =1}$. 

\begin{theorem}[Bender's Lemma \cite{bender}]
Suppose that $A(z)=\sum a_n z^n$ and $B(z)=\sum b_n z^n$ are power series with radii of convergence $\alpha> \beta \geq 0$, respectively. Suppose $b_{n-1}/b_n$ approaches a limit $b$ as $n$ approaches infinity. If $A(b) \neq 0$, then $c_n \sim A(b)b_n$, where $\sum c_nz^n=A(z)B(z)$.
\end{theorem}

\begin{lemma}
For all $k \geq 1$, $d_k(x)$ is never zero on the closed disk of radius $\frac 4 {15}$. 
\end{lemma}
\begin{proof}
Let $|x| \leq \frac 4 {15}$. Then $|d_1(x)|=1 > 0$, $|d_2(x)| \geq 2 - |x| >0$, and
$$|d_3(x)| \geq 1 - |3x| - |2x^2| - |x^3| \geq 1 - \frac 4 5 - \frac {32}{225} - \frac {192}{3375}=\frac{131}{3375}.$$
Observe that, for $k \geq 2$, $|d_k(x)-1| <\left |x^{k-2} + \sum_{k=1}^\infty (k+1)x^k \right |$. Now we have that 
%\begin{equation}
$$|d_k(x)-1| <\left |x^{k-2}+\sum_{k=1}^\infty (k+1)x^k \right| $$
$$\leq \left ( \frac{4}{15} \right)^{k-2} + \sum_{k=1}^{\infty} (k+1)\left( \frac{4}{15} \right)^k =\left( \frac{4}{15} \right)^{k-2}+\frac{165}{196}.$$
%\end{equation}
If $k \geq 4$, then $\left( \frac{4}{15} \right)^{k-2} < \frac {31}{196}$, so that $|d_k(x)-1| <1$. Hence \linebreak $|d_k(x)| \geq 1 - |d_k(x)-1| >0$ for all $k \geq 4.$
\end{proof}
\begin{theorem}
Let $t_k(n)$ be the number of vertices which are $k$-protected in all rooted plane trees of size of $n$. Then $t_k(n) \sim \frac{3{2n -2\choose n-1}}{(4^k+2)}$.
\end{theorem}

\begin{proof} Note that $\frac{x}{\sqrt{1-4x}}$ is the generating function for the number of all vertices of all trees of size $n$. We have that 
%\begin{equation}
$$T_k(x=L(x) \cdot R_k(x)=\left(\frac{1+\sqrt{1-4x}}{2\sqrt{1-4x}} \right) \cdot \left (  \frac{x^{k-2}\left( n_k(x) -\sqrt{1-4x}\right)}{2d_k(x)} \right) $$
$$=x^{k-2} \cdot \frac{n_k(x)-(1-4x)+(n_k(x)-1)\sqrt{1-4x}}{4 d_k(x) \sqrt{1-4x}} $$
$$=\frac{x}{\sqrt{1-4x}} \cdot \frac{x^{k-3}(n_k(x)-1+4x)}{4d_k(x)} -\frac{ x^{k-2}(n_k(x)+1)}{4 d_k(x)}$$
%\end{equation}
Since $d_k(x)$ is non-zero on the closed disk of radius 4/15, it follows that the term on the right is asymptotically irrelevant and that $\frac{x^{k-3}(n_k(x)-1+4x)}{4d_k(x)}$ has a radius of convergence larger than 1/4 for all $k$. Thus we can apply Bender's Lemma to the term on the left for any $k$, giving
\begin{equation}
[x^n] T_k(x) \sim {2n -2\choose n-1} \frac{(1/4)^{k-3}(n_k(1/4)-1+4(1/4))}{4d_k(1/4)}.
\end{equation}
Since we have the recurrence $n_{k+1}(1/4)=n_k(1/4)-2(1/4)^{k+1}$ and $n_1(1/4)=3/4$, we have $n_k(1/4)=\frac{2+4^k}{3 \cdot 4^k}$. Applying Lemma 3 we have
$$n_k^2(1/4)-(1-4(1/4))=4(1/4)^3 d_k(1/4),$$
so that $d_k(1/4)=16 n_k^2(1/4)$.
Thus
$$\frac{(1/4)^{k-3}(n_k(1/4)-1+4(1/4))}{4d_k(1/4)}=\frac{1}{4^{k} \cdot \frac{2+4^k}{3 \cdot 4^k}}=\frac{3}{4^k+2}.$$

\end{proof}

\begin{corollary}
Let $p_k(n)$ be the probability that a random vertex in a random rooted plane tree of size $n$ is $k$-protected. Then $p_k(n) \sim \frac{3}{4^k+2}$.
\end{corollary}
\begin{proof} To find the average, we divide by ${2n -2\choose n-1}$.
\end{proof}
This gives the sequence of values 1, 1/2, 1/6, 1/22, 1/86,..., and also shows that as we progress to a higher level of protection, we lose about 1/4 of the vertices each time. The result for $k=2$ agrees with the result by Cheon and Shapiro \cite{cheon}.

For numerical justification, we have, if $k=3$,
$$ \frac{[x^{50}] T_3(x)}{[x^{50}] \frac {x}{\sqrt{1-4x}}}=\frac {88 972 411 304 864 387 146 864 997}{1959816327613912069440802200} \approx 0.0453986$$
and $\frac 1{22} =0.0\overline{45}$.

There are some other interesting questions which can be answered using $R_k(x)$. The height of a tree is defined as the longest path from the root to a leaf. The function $R_k(x)$ enumerates instead by the \textit{shortest} path from the root to a leaf.

\begin{theorem}
Let $r_k(n)$ be the number of trees on $n$ vertices whose root is $k$-protected. Then

$$r_k(n) \sim \frac{9}{4^{1-k}+4+4^k}\cdot c_{n-1}$$
where $c_{n-1}$ denotes the $(n-1)$st Catalan number.
\end{theorem}
\begin{proof}
We have that $$R_k(x)=\frac{x^{k-2}(1-\sqrt{1-4x})}{2d_k(x)}-  \frac{\sum_{i=1}^k x^{i+k-2}}{d_k(x)}.$$
The term on the right is once again asymptotically irrelevant, so we have
$$R_k(x)=\frac{x^{k-2}}{d_k(x)} \cdot T(x)-\frac{\sum_{i=1}^k x^{i+k-2}}{d_k(x)}.$$
Thus we have
$$\frac{(1/4)^{k-2}}{d_k(1/4)}\cdot  c_{n-1}=\frac{(1/4)^{k-2}}{16 n_k^2(1/4)}\cdot  c_{n-1}=\frac{1}{4^k \left( \frac{2+4^k}{3 \cdot 4^k}\right)^2}\cdot  c_{n-1}$$
$$=\frac{9}{4^{-k}(4+4^{k+1}+4^{2k})} \cdot c_{n-1}=\frac{9}{4^{1-k}+4+4^k}\cdot c_{n-1}.$$
\end{proof}
This gives the sequenced of values 1, 1, 4/9, 16/121, 64/1849, ..., and once again shows that we lose about 1/4 of the trees each time we progress to a higher level of protection.

Recall that the \textit{rank} of a vertex is the distance of the shortest downward path from a vertex to a leaf.

\begin{corollary}
The number of vertices of rank $k$ in trees of size $n$ approaches the value $\frac{ 9 {2n -2\choose n-1}}{10+4^{1-k}+4^{1+k}}$.
\end{corollary}

\begin{proof}
It follows immediately from the definition that a vertex has rank $k$ if and only if it is $k$-protected but not $(k+1)$-protected, so by Theorem 3, the number of vertices of rank $k$ approaches
$$\frac{3 {2n -2\choose n-1}}{(4^k+2)}-\frac{3 {2n -2\choose n-1}}{(4^{k+1}+2)}=3 {2n -2\choose n-1} \cdot \left( \frac{4^{k+1}+2-4^k-2}{4^{2k+1}+2(4^k)+2(4^{k+1})+4}\right)$$
$$=3 {2n -2\choose n-1} \cdot \frac{4^k (4-1)}{4^k(4^{k+1}+2+2(4)+4^{1-k})}=\frac{ 9 {2n -2\choose n-1}}{10+4^{1-k}+4^{1+k}}.$$
\end{proof}

\section{Expectations}

\begin{theorem}
Let $E_R(n)$ denote the expected value of the rank of the root of a tree of size n. Then
\begin{equation}
E_R(n) \sim \sum_{k=1}^\infty \frac{9}{4^{1-k}+4+4^k} \approx  1.62297.
\end{equation}
\end{theorem}

\begin{theorem}
Let $E_T(n)$ denote the expected value of the rank of a random vertex in a tree of size n. Then
\begin{equation}
E_T(n) \sim \sum_{k=1}^\infty \frac{3}{4^k+2} \approx 0.727649.
\end{equation}
\end{theorem}

\begin{proof}[Proof of Theorems 5]
Let $E_R(x)=\sum_{n=1}^\infty r(n) x^n$ where $r(n)$ is the sum of the ranks of the roots of all trees of size $n$. Clearly $r(n)= \sum_{k=1}^n [x^n] R_k(x)$, since the trees with roots of rank one is counted once by $R_1(x)$, the trees with roots of rank two are counted once by $R_1(x)$ and once by $R_2(x)$, and similarly the trees with roots of rank $n$ are counted once by each $R_k(x)$ for $1 \leq k \leq n$. Since $R_k(x)$ has no terms of degree lower than $k$, their sum converges as a formal power series. It follows that 
$$E_R(x)= \sum_{k=1}^\infty R_k(x)=\sum_{k=1}^\infty \left( \frac{x^{k-2}(1-\sqrt{1-4x})}{2d_k(x)}-\frac{\sum_{i=1}^{k-2} x^{i+k-2}}{d_k(x)}\right) .$$
The sums $\sum_{k=1}^\infty \frac{x^{k-2}}{d_k(x)}$ and $\sum_{k=1}^\infty \frac{\sum_{i=1}^{k-2}x^{i+k-2}}{d_k(x)}$ both converge absolutely on the closed disk of radius 1/4, so we can split the sum as follows: 
%\begin{equation}
$$\sum_{k=1}^\infty R_k(x) = \sum_{k=1}^\infty \left( \frac{x^{k-2}}{d_k(x)} \cdot T(x) - \frac{\sum_{i=1}^{k-2}x^{i+k-2}}{d_k(x)} \right) $$
$$= \sum_{k=1}^\infty \frac{x^{k-2}}{d_k(x)} \cdot T(x) - \sum_{k=1}^\infty \frac{\sum_{i=1}^{k-2}x^{i+k-2}}{d_k(x)}.$$
%\end{equation}

Since the coefficient of $T(x)$ and the double sum on the right both converge absolutely to bounded, analytic functions on the closed disk of radius 5/14, it follows that the term on the right remains asymptotically irrelevant, and that we may apply Bender's Lemma to the term on the left, and the result follows.

\end{proof}
\begin{proof}[Proof of Theorem 6]
The proof is similar to that of Theorem 5 with the role of $R_k$ replaced by $T_k$.
\end{proof}

For numerical justification, we may compute the $n$th coefficient using only the $(n-1)$st partial sum, since there are no vertices or roots which are $n$-protected in a tree of size $n$. We have that
$$\frac{[x^{50}]\sum_{k=1}^{49}R_k(x)}{[x^{50}]T(x)}=\frac{1874097069430998779470999}{1152833133890536511435766} \approx 1.62564$$

and
$$\frac{[x^{50}] \sum_{k=1}^{49} T_k(x)}{[x^{50}] V(x)}=\frac{4630522930774422812075437903}{6369403064745214225682607150} \approx 0.726995.$$

%I think I have the variance, for vertices it should be 0.816899 and for roots 0.715698 but proving that the series converge seems iffy...

\bibliographystyle{amsplain}
\bibliography{bib}
\end{document}